\documentclass[11pt, a4paper]{amsart}

\usepackage[hmargin=32mm, vmargin=27mm, includefoot, twoside]{geometry}
\usepackage[bookmarksopen=true]{hyperref}
\usepackage[english]{babel}
\usepackage{amsmath,amssymb}
\usepackage{mathrsfs}
\usepackage{dsfont}
\usepackage{lmodern}
\usepackage{graphicx}
\usepackage[utf8]{inputenc}
\usepackage[figurename=Fig.]{caption}

\newtheorem{thmintro}{Theorem}

\newtheorem{corintro}[thmintro]{Corollary}
\newtheorem{theorem}{Theorem}[section]

\newtheorem{lemma}[theorem]{Lemma}
\newtheorem{prop}[theorem]{Proposition}
\theoremstyle{definition}
\newtheorem{remark}[theorem]{Remark}
\newtheorem{example}[theorem]{Example}
\newtheorem{definition}[theorem]{Definition}

\newcommand{\NN}{\mathbb{N}}
\newcommand{\RR}{\mathbb{R}}
\newcommand{\CCC}{\mathscr{C}}
\newcommand{\OOO}{\mathcal{O}}
\newcommand{\inv}{^{-1}}
\newcommand{\co}{\colon\thinspace}
\newcommand{\ra}{\rightarrow}
\newcommand{\tc}{\approx_t}
\newcommand{\ras}[1]{\stackrel{#1}{\rightarrow}}
\newcommand{\simx}[1]{\stackrel{#1}{\sim}}

\DeclareMathOperator{\CAT}{CAT(0)}
\DeclareMathOperator{\Ch}{Ch}
\DeclareMathOperator{\CMin}{CombiMin}
\DeclareMathOperator{\dist}{d}
\DeclareMathOperator{\dc}{\dist_{\Ch}}
\DeclareMathOperator{\Min}{Min}
\DeclareMathOperator{\Pc}{Pc}
\DeclareMathOperator{\proj}{proj}
\DeclareMathOperator{\Stab}{Stab}
\DeclareMathOperator{\supp}{supp}

\numberwithin{equation}{section}

\begin{document}

\renewcommand{\proofname}{{\bf Proof}}

\title[Cyclically reduced elements in Coxeter groups]{Cyclically reduced elements in Coxeter groups}


\author[T.~Marquis]{Timoth\'ee \textsc{Marquis}$^*$}
\address{UCLouvain, IRMP, 1348 Louvain-la-Neuve, Belgium}
\email{timothee.marquis@uclouvain.be}
\thanks{$^*$F.R.S.-FNRS Research Fellow}

\subjclass[2010]{20F55, 20E45}

\begin{abstract}
Let $W$ be a Coxeter group. We provide a precise description of the conjugacy classes in $W$, in the spirit of Matsumoto's theorem. This extends to all Coxeter groups an important result on finite Coxeter groups by M.~Geck and G.~Pfeiffer from 1993. In particular, we describe the cyclically reduced elements of $W$, thereby proving a conjecture of A.~Cohen from 1994.
\end{abstract}




\maketitle

\section{Introduction}
Let $(W,S)$ be a Coxeter system. By a classical result of J.~Tits (\cite{Tit69}), also known as Matsumoto's theorem (see \cite{Mat64}), any given reduced expression of an element $w\in W$ can be obtained from any other expression of $w$ by performing a finite sequence of braid relations and $ss$-cancellations (i.e. replacing a subword $(s,s)$ for some $s\in S$ by the empty word). In particular, this yields a very simple and elegant solution to the word problem in Coxeter groups.

The conjugacy problem for Coxeter groups was solved about 30 years later, by D.~Krammer in his thesis from 1994 (published in \cite{Kra09}): there exists a cubic algorithm deciding whether two words on the alphabet $S$ determine conjugate elements of $W$. However, Krammer's solution does not provide a sequence of ``elementary operations'' to pass from one word to the other, as do the braid relations and $ss$-cancellations in Matsumoto's theorem.

In this paper, we address the following long-standing open question on Coxeter groups: \emph{Is there an analogue of Matsumoto's theorem for the conjugacy problem in Coxeter groups?}

A very natural elementary operation on words to consider for the conjugacy problem is that of cyclic shift: by extension, we say that an element $w'\in W$ is a {\bf cyclic shift} of some $w\in W$ if there is some reduced decomposition $w=s_1\dots s_d$ ($s_i\in S$) of $w$ such that either $w'=s_2\dots s_ds_1$ or $w'=s_ds_1\dots s_{d-1}$. Such operations are, however, not sufficient to describe conjugacy classes in general, as for instance illustrated by the Coxeter group $W=\langle s,t \ | \ s^2=t^2=(st)^3=1\rangle$ of type $A_2$, in which the simple reflections $s$ and $t$ are conjugate, but cannot be obtained from one another through a sequence of cyclic shifts. Nonetheless, in the terminology of \cite[Chapter~3]{GP00},  the elements $w:=s$ and $w':=t$ are {\bf elementarily strongly conjugate}, meaning that $\ell_S(w)=\ell_S(w')$ and that there exists some $x\in W$ with $w'=x\inv wx$ such that either $\ell_S(x\inv w)=\ell_S(x)+\ell_S(w)$ or $\ell_S(wx)=\ell_S(w)+\ell_S(x)$.

Motivated by the representation theory of Hecke algebras, M.~Geck and G.~Pfeiffer proved in \cite{GP93} that if $W$ is finite, then for any conjugacy class $\OOO$ in $W$,
\begin{enumerate}
\item
any $w\in \OOO$ can be transformed by cyclic shifts into an element $w'$ of minimal length in $\OOO$, and
\item
any two elements $w,w'$ of minimal length in $\OOO$ are {\bf strongly conjugate}, i.e. there exists a sequence $w=w_0,\dots,w_n=w'$ of elements of $W$ such that $w_{i-1}$ is elementarily strongly conjugate to $w_i$ for each $i=1,\dots,n$.
\end{enumerate}
Together with S.~Kim, they later generalised this theorem (see \cite{GKP00}) to the case of $\delta$-twisted conjugacy classes for some automorphism $\delta$ of $(W,S)$, that is, when $\OOO$ is replaced by $\OOO_{\delta}=\{\delta(v)\inv wv \ | \ v\in W\}$ for some $w\in W$. The proofs in \cite{GP93} and \cite{GKP00} involve a case-by-case analysis, with the help of a computer for the exceptional types. In \cite{HN12}, X.~He and S.~Nie gave a uniform (and computer-free) geometric proof of that theorem, which they later generalised, in \cite{HN14}, to the case of an affine Coxeter group $W$. In addition, they showed (for $W$ affine) that 
\begin{enumerate}
\item[(3)]
if $\OOO$ is straight, then any two elements $w,w'$ of minimal length in $\OOO$ are conjugate by a sequence of cyclic shifts,
\end{enumerate}
where $\OOO$ is {\bf straight} if it contains a straight element $w\in W$, that is, such that $\ell_S(w^n)=n\ell_S(w)$ for all $n\in\NN$ (equivalently, every minimal length element of $\OOO$ is straight, see Lemma~\ref{lemma:straight_lemma_4.2}). Note that the straight elements in an arbitrary Coxeter group were characterised in \cite[Theorem~D]{straight}; these elements play an important role in the study of affine Deligne-Lusztig varieties (see \cite{He14}), and also exhibit very useful dynamical properties (see e.g. \cite{simpleKM} or \cite{CH15}). Similar statements to (1) and (2) above were further obtained for an arbitrary Coxeter group $W$, but when $\OOO$ is replaced by some ``partial'' conjugacy class $\OOO=\{v\inv wv \ | \ v\in W_I\}$, for some finite standard parabolic subgroup $W_I\subseteq W$ (see \cite{He07} and \cite{Nie13}). Finally, we showed in \cite[Theorem~A]{straight} that for a certain class of Coxeter groups that includes the right-angled ones, (1) and (2) hold using only cyclic shifts.

\medskip

In this paper, we prove the statements (1), (2) and (3) in full generality, namely, for an arbitrary Coxeter group $W$. Moreover, we actually prove a much more precise version of (2) by introducing a refined notion of ``strong conjugation'', which we call ``tight conjugation'' (see Definition~\ref{definition:tightconj}) --- in particular, if two elements are tightly conjugate, then they are strongly conjugate; when $W$ is finite, the two notions coincide. Here is our main theorem.

\begin{thmintro}\label{thmintro:mainthm}
Let $(W,S)$ be a Coxeter system. Let $\OOO$ be a conjugacy class of $W$, and let $\OOO_{\min}$ be the set of minimal length elements of $\OOO$. Then the following assertions hold:
\begin{enumerate}
\item
For any $w\in\OOO$, there exists an element $w'\in\OOO_{\min}$ that can be obtained from $w$ by a sequence of cyclic shifts.
\item
If $w,w'\in\OOO_{\min}$, then $w$ and $w'$ are tightly conjugate.
\item
If $\OOO$ is straight, then any two elements $w,w'\in\OOO_{\min}$ are conjugate by a sequence of cyclic shifts.
\end{enumerate}
\end{thmintro}

Note that the proof of Theorem~\ref{thmintro:mainthm} uses the results of \cite{GKP00} (or \cite{HN12}), but does not rely on \cite{HN14}. In particular, we give an alternative, shorter proof that affine Coxeter groups satisfy Theorem~\ref{thmintro:mainthm}.

Recall that an element $w\in W$ is {\bf cyclically reduced} if $\ell_S(w')=\ell_S(w)$ for every $w'\in W$ obtained from $w$ by a sequence of cyclic shifts. Often, this terminology is used instead for elements of minimal length in their conjugacy class. An important reformulation of Theorem~\ref{thmintro:mainthm}(1) is that these two notions in fact coincide. 
\begin{corintro}
An element $w\in W$ is cyclically reduced if and only if it is of minimal length in its conjugacy class.
\end{corintro}
This proves a conjecture of A.~Cohen (see \cite[Conjecture 2.18]{Coh94}).

The proof of Theorem~\ref{thmintro:mainthm} is of geometric nature, and uses the Davis complex $X$ of $(W,S)$ --- here, we assume that $S$ is finite, a safe assumption for the study of Theorem~\ref{thmintro:mainthm} (see Remark~\ref{remark:finite_rank_reduction}). This is a $\CAT$ cellular complex on which $W$ acts by cellular isometries. For instance, if $W$ is affine, then $X$ is just the standard geometric realisation of the Coxeter complex $\Sigma$ of $(W,S)$, and the $\CAT$ metric $\dist\co X\times X\to\RR_+$ is the usual Euclidean metric (see Example~\ref{example:Davis_complex}). For an element $w\in W$, the subset $\Min(w)\subseteq X$ of all $x\in X$ such that $\dist(x,wx)$ is minimal will play an important role; it will also be crucial to investigate its combinatorial analogue $\CMin(w)$ (see \S\ref{section:TCC}), as highlighted in Remark~\ref{remark:CMin_in_Cw}. As a byproduct of our proofs, we are able to relate these two notions of ``minimal displacement set'' for $w$ (see \S\ref{section:COMAC}).
\begin{corintro}\label{corintro:comparison_Min}
Let $w\in W$. Then $\Min(w)\subseteq\CMin(w)$, and $\CMin(w)$ is at bounded Hausdorff distance from $\Min(w)$.
\end{corintro}

Note that, while $\Min(w)$ is always connected (in the $\CAT$ sense), its combinatorial analogue $\CMin(w)$ need not be (gallery-)connected, and this is precisely the reason why cyclic shifts are not sufficient to describe the conjugacy classes in $W$, and why one needs to also consider ``tight conjugations'' (see Remark~\ref{remark:CMin_not_convex}).

\medskip

We conclude the introduction with a short roadmap to the proof of Theorem~\ref{thmintro:mainthm}. Let $w\in W$, and let $\OOO$ (resp. $\OOO_{\min}$) denote its conjugacy class (resp. the set of elements of minimal length in $\OOO$). Let $\Ch(\Sigma)$ be the set of chambers of the Coxeter complex $\Sigma$ ($\Ch(\Sigma)$ can be $W$-equivariantly identified with the set of vertices $\{v \ | \ v\in W\}$ of the Cayley graph of $(W,S)$). 

The first step is to view the elements of $\OOO$ and $\OOO_{\min}$ geometrically, as chambers of $\Sigma$: we consider the map $\pi_w\co \Ch(\Sigma)\to W: v\mapsto v\inv w v$, which satisfies $\pi_w(\Ch(\Sigma))=\OOO$ and $\pi_w\inv(\OOO_{\min})=\CMin(w)$. The second step is to interpret the operations of cyclic shifts and tight conjugations geometrically, at the level of $\Ch(\Sigma)$, by defining two ``elementary geometric operations'', say of type (I) and (II), allowing to pass from one chamber $C\in\Ch(\Sigma)$ to another chamber $D\in\Ch(\Sigma)$, in such a way that passing from $C$ to $D$ with an operation of type (I) implies that one can pass from $\pi_w(C)$ to $\pi_w(D)$ using cyclic shifts, and passing from $C$ to $D$ with an operation of type (II) implies that one can pass from $\pi_w(C)$ to $\pi_w(D)$ using tight conjugations; this strategy is implemented in Sections~\ref{section:TCC} and \ref{section:TCCw}, and makes use of the analogue of Theorem~\ref{thmintro:mainthm} for twisted conjugacy classes in finite Coxeter groups established in \cite{GKP00} and \cite{HN12}. Theorem~\ref{thmintro:mainthm} then amounts to showing that one can pass from the chamber $C_0:=\{1_W\}$ (representing $\pi_w(C_0)=w$) to any other chamber $C$ by a sequence of geometric operations of type (I) and (II) (see Section~\ref{section:TCPIWS}). This geometric formulation of the problem allows one to take advantage of the tools provided by $\CAT$ geometry, and of the specific properties of Davis complexes described in \S\ref{subsection:DC}--\ref{subsection:W}. Finally, note that the analogue of Theorem~\ref{thmintro:mainthm} for untwisted conjugacy classes in finite Coxeter groups, first established in \cite{GP93}, is also used to prove Theorem~\ref{thmintro:mainthm} when $w$ has finite order.

\section{Preliminaries}
\subsection{Basic definitions}\label{subsection:BD}
Basics on Coxeter groups and complexes can be found in \cite[Chapters~1--3]{BrownAbr}. The notions introduced below are illustrated on Figure~\ref{figure:A2tilde} (see Example~\ref{example:A2tilde}).

Throughout this paper, $(W,S)$ denotes a Coxeter system of finite rank (see Remark~\ref{remark:finite_rank_reduction}). We let $\Sigma=\Sigma(W,S)$ be the associated Coxeter complex, with set of roots (or half-spaces) $\Phi$. Let also $C_0:=\{1_W\}$ be the fundamental chamber of $\Sigma$, and $\Pi:=\{\alpha_s \ | \ s\in S\}$ be the corresponding set of simple roots (i.e. the roots containing $C_0$ and whose wall is a wall of $C_0$). Write $\Ch(\Sigma):=\{wC_0 \ | \ w\in W\}$ for the set of chambers of $\Sigma$. We will often identify a chamber subcomplex $A$ of $\Sigma$ with its underlying set $\Ch(A)\subseteq\Ch(\Sigma)$ of chambers.

Two chambers $D,E\in\Ch(\Sigma)$ are {\bf $s$-adjacent} for some $s\in S$ if they are $s$-adjacent in the Cayley graph $\mathrm{Cay}(W,S)$. A {\bf gallery} $\Gamma$ between two chambers $D,E\in\Ch(\Sigma)$ is a sequence of chambers $D=D_0,D_1,\dots,D_r=E$ such that, for each $i\in\{1,\dots,r\}$, the chamber $D_{i-1}$ is (distinct from and) $s_i$-adjacent to $D_i$ for some $s_i\in S$. The sequence $(s_1,\dots,s_r)\in S^r$ is the {\bf type} of $\Gamma$, and $\ell(\Gamma):=r$ the {\bf length} of $\Gamma$. The gallery $\Gamma$ is {\bf minimal} if it is a gallery of minimal length between $C$ and $D$; in this case, we set $\dc(C,D):=\ell(\Gamma)$ and call it the {\bf chamber distance} between $C$ and $D$. We let
$$\Gamma(C,D):=\{E\in\Ch(\Sigma) \ | \ \dc(C,D)=\dc(C,E)+\dc(E,D)\}$$
denote the set of chambers on a minimal gallery from $C$ to $D$. If $C=vC_0$ and $D=wvC_0$ for some $v,w\in W$, there is a bijective correspondence between minimal galleries $\Gamma$ from $C$ to $D$ and reduced expressions (on the alphabet $S$) for $v\inv wv$, mapping the gallery $\Gamma$ to its type. In particular, if we again denote by $\ell\co W\to \NN$ the word length on $W$ with respect to $S$, then $\ell(v\inv wv)$ coincides with $\ell(\Gamma)$, or else with the number of walls crossed by $\Gamma$ (i.e. the number of walls separating $C$ from $D$).

To each simplex $\sigma$ of $\Sigma$, one associates its corresponding {\bf residue} $R_{\sigma}$, which is the set of chambers of $\Sigma$ containing $\sigma$. A {\bf wall of $R_{\sigma}$} is a wall of $\Sigma$ containing $\sigma$. For a subset $I\subseteq S$, we let $W_I:=\langle I\rangle\subseteq W$ denote the {\bf standard parabolic subgroup} of type $I$. The {\bf parabolic subgroups} of $W$ are then the conjugates of the standard parabolic subgroups, or equivalently, the stabilisers in $W$ of some simplex (resp. residue) of $\Sigma$. The simplex $\sigma$ (resp. the residue $R_{\sigma}$) is {\bf spherical} if its stabiliser $P_{\sigma}$ in $W$ is finite; it is {\bf standard} if $P_{\sigma}$ is a standard parabolic subgroup (equivalently, if $\sigma$ is a face of $C_0$, resp. if $C_0\in R_{\sigma}$). Thus, if $\sigma$ is a face of $vC_0$ for some $v\in W$, then $P_{\sigma}=vW_Iv\inv$ for some subset $I\subseteq S$. For any $w\in W$, there is a smallest parabolic subgroup $\Pc(w)$ containing $w$, called the {\bf parabolic closure} of $w$.

For each $I\subseteq S$, we set $\Pi_I:=\{\alpha_s \ | \ s\in I\}$. Let $N_I$ be the stabiliser in $W$ of $\Pi_I$. Note that the conjugation action of any $n_I\in N_I$ on $W_I$ induces an automorphism of $W_I$ preserving $I$ (called a {\bf diagram automorphism}). We write $N_W(W_I)$ for the normaliser of $W_I$ in $W$, and we call $I$ {\bf spherical} if $W_I$ is finite. The following lemma follows from \cite[Lemma~5.2]{Lus77}. 
\begin{lemma}\label{lemma:Lustzig}
Let $I\subseteq S$. Then $N_W(W_I)=W_I\rtimes N_I$. Moreover, 
\begin{equation}\label{eqn:Lustzig}
\ell(w_In_I)=\ell(w_I)+\ell(n_I)\quad\textrm{for all $w_I\in W_I$ and $n_I\in N_I$.}
\end{equation}
\end{lemma}

\begin{figure}

  \centering
  \includegraphics[trim = 37mm 195mm 93mm 27mm, clip, width=9cm]{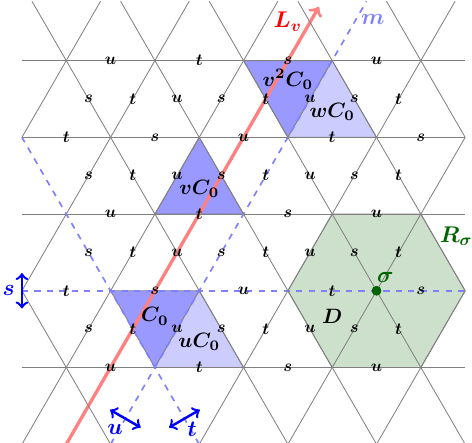}
  \captionof{figure}{Coxeter complex of type $\widetilde{A}_2$}
  \label{figure:A2tilde}

\end{figure}

\begin{example}\label{example:A2tilde}
For the benefit of the reader unfamiliar with Coxeter groups and complexes, we illustrate the above notions on an example.
Consider the (affine) Coxeter group $W=W(\widetilde{A}_2)$ of type $\widetilde{A}_2$, with standard generating set $S=\{s,t,u\}$ and presentation $W=\langle s,t,u \ | \ s^2=t^2=u^2=1=(st)^3=(su)^3=(tu)^3\rangle$.

The Coxeter complex $\Sigma=\Sigma(W,S)$ is then the simplicial complex associated to the tesselation of the Euclidean plane by congruent equilateral triangles (see Figure~\ref{figure:A2tilde}). Fixing a fundamental chamber (i.e. a triangle) $C_0$, the generators $s,t,u$ act as orthogonal reflections across the walls (i.e. lines) delimiting $C_0$ (see the dashed lines on Figure~\ref{figure:A2tilde}). The $W$-action on $\Sigma$ is then simply transitive of the set $\Ch(\Sigma)$ of chambers. The roots $\alpha_s,\alpha_t,\alpha_u$ are the half-spaces respectively delimited by the walls of $s,t,u$ (i.e. the lines fixed by $s,t,u$) and containing $C_0$.

To each codimension $1$ face of $C_0$ (i.e. edge of $C_0$, say contained in the wall of $x\in S$), one attributes its type $x\in S$. One then extends this labelling to all edges by requiring the $W$-action to be type-preserving. Two distinct chambers are then $x$-adjacent if they share a common edge of type $x$. Together with these $x$-adjacency relations, $\Ch(\Sigma)$ then coincides with the Cayley graph of $(W,S)$.

One can reconstruct $\Sigma$ group-theoretically by attaching to each face of $C_0$ its stabiliser in $W$: the chamber $C_0$ (together with its faces) is isomorphic to the poset of standard parabolic subgroups, ordered by the opposite of the inclusion relation (indeed, $\Stab_W(C_0)=W_{\varnothing}=\{1\}$, the stabiliser of the edge of $C_0$ labelled $x\in S$ is $W_{\{x\}}$, the stabiliser of the vertex of $C_0$ at the intersection of the edges labelled $x$ and $y$ is $W_{\{x,y\}}$, and the stabiliser of the empty simplex is $W$). Using the $W$-action, $\Sigma$ can thus be defined as the poset $\{wW_I \ | \ w\in W, \ I\subseteq S\}$, ordered by the opposite of the inclusion relation (the $W$-action being by left translation).

An example of $0$-dimensional simplex (i.e. vertex) $\sigma$, as well as the corresponding residue $R_{\sigma}$, are pictured on Figure~\ref{figure:A2tilde}. Since $D=ustuC_0\in R_{\sigma}$, the stabiliser of $R_{\sigma}$ is the parabolic subgroup $P_{\sigma}=(ustu)W_{\{s,t\}}(ustu)\inv$.

Finally, note that the element $w:=sutsutu$ commutes with $u$, and hence $w\in N_W(W_I)$ for $I=\{u\}$. The decomposition $w=w_In_I$ with $w_I\in W_I$ and $n_I\in N_I$ provided by Lemma~\ref{lemma:Lustzig} is then given by $w_I=u$ and $n_I=sutsut$. 
\end{example}

\subsection{Straight elements}
An element $w\in W$ is called {\bf straight} if $\ell(w^n)=n\ell(w)$ for all $n\in\NN$. We record for future reference the following basic properties of straight elements.

\begin{lemma}[{\cite[Lemma~4.1]{straight}}]\label{lemma:straight_lemma_4.1}
Let $w\in W$ be straight. Then $w$ is of minimal length in its conjugacy class. Moreover, if $w\in N_W(W_I)$ for some spherical subset $I\subseteq S$, then $w\in N_I$.
\end{lemma}

\begin{lemma}[{\cite[Lemma~4.2]{straight}}]\label{lemma:straight_lemma_4.2}
Let $v,w\in W$ be such that $\ell(v\inv w v)=\ell(w)$. Then $w$ is straight if and only if $v\inv wv$ is straight.
\end{lemma}

\subsection{Projections}
The general reference for this section is \cite[Chapter~21]{MPW15}. 
Given a chamber $D\in\Ch(\Sigma)$ and a residue $R$, there is a unique chamber $E\in R$ at minimal distance from $D$, called the {\bf projection} of $D$ on $R$, and denoted $\proj_R(D)$. Alternatively, $\proj_R(D)$ is the unique chamber $E$ of $R$ such that $D$ and $E$ lie on the same side of every wall of $R$. In particular, one has the following {\bf gate property}:
$$\dc(D,E)=\dc(D,\proj_R(D))+\dc(\proj_R(D),E)\quad\textrm{for any chamber $E\in R$.}$$
As $\proj_R\co\Ch(\Sigma)\to R$ maps galleries to galleries, it does not increase the chamber distance.
Two residues $R,R'$ are {\bf parallel} if the projection map $\proj_R|_{R'}\co R'\to R$ is bijective (in which case $\proj_{R'}|_{R}\co R\to R'$ is its inverse). Equivalently, $R$ and $R'$ are parallel if and only if they have the same walls if and only if they have the same stabiliser in $W$. In that case, $\dc(D,\proj_{R}(D))$ is independent of the choice of chamber $D\in R'$. Finally, if $R\subseteq R'$, then $\proj_{R}(D)=\proj_R(\proj_{R'}(D))$ for all $D\in\Ch(\Sigma)$.

\begin{example}
Keeping the notations of Example~\ref{example:A2tilde}, the projection of $C_0$ on the residue $R_{\sigma}$ is the chamber $D$. Let now $\tau_1$ be the common edge shared by $C_0$ and $uC_0$, and let $\tau_2$ be the common edge shared by $wC_0$ and $uwC_0$ (also denoted $v^2C_0$ on Figure~\ref{figure:A2tilde}, where $v:=sut$). Then the residues $R_{\tau_1}=\{C_0,uC_0\}$ and $R_{\tau_2}=\{wC_0,uwC_0\}$ have the same walls (i.e. the wall $m$ of $u$), and are therefore parallel. The bijection between $R_{\tau_1}$ and $R_{\tau_2}$ induced by the projections identifies $C_0$ with $uwC_0$ and $uC_0$ with $wC_0$.
\end{example}

\begin{example}\label{example:RwR}
Let $R$ be a residue, and assume that $R$ and $wR$ are parallel for some $w\in W$. Then $w$ normalises $\Stab_W(R)$. If, moreover, $R$ is spherical and standard, so that $\Stab_W(R)=W_I$ for some spherical subset $I\subseteq S$, then $w=w_In_I$ for some $w_I\in W_I$ and $n_I\in N_I$ by Lemma~\ref{lemma:Lustzig}. By definition of $N_I$, we then have $\proj_{wR}(C_0)=n_IC_0$.
\end{example}

\subsection{Davis complex}\label{subsection:DC}
The general reference for this section is \cite{Davis}. We briefly recall the construction of the {\bf Davis complex} $X$ of $(W,S)$, which is a complete, uniquely geodesic metric realisation of $\Sigma$. Let $\Sigma_{(1)}$ be the flag complex of $\Sigma$, that is, $\Sigma_{(1)}$ is the simplicial complex with vertices the simplices of $\Sigma$ and simplices the flags of simplices of $\Sigma$. Let also $\Sigma_{(1)}^{s}$ denote the subcomplex of $\Sigma_{(1)}$ with vertices the spherical simplices of $\Sigma$. Then $X$ is the geometric realisation of $\Sigma_{(1)}^s$ (hence a cellular subcomplex of the barycentric subdivision of the geometric realisation of $\Sigma$), together with a suitably defined $\CAT$ metric $\dist\co X\times X\to\RR_+$ extending the canonical Euclidean metrics on its cells. Each (open) cell $\sigma$ of $X$ corresponds to a unique spherical simplex $wW_I$ of $\Sigma$ --- namely, $\sigma$ is (the realisation of) the union of all flags of spherical simplices whose upper bound is $wW_I$ --- and the $W$-action on the spherical simplices of $\Sigma$ induces a cellular isometric $W$-action on $X$. 

For each $x\in X$, there is a unique (open) cell $\supp(x)$ containing $x$, called the {\bf support} of $x$. In particular, $\Stab_W(x)=\Stab_W(\supp(x))$ is a spherical (i.e. finite) parabolic subgroup of $W$. In this paper, we shall identify the roots, walls and chambers of $\Sigma$ with the corresponding \emph{closed} subsets of $X$. In particular, a chamber $D\in\Ch(\Sigma)\approx\Ch(X)$ will be identified with the set of $x\in X$ whose support corresponds either to $D$ \emph{or to a (spherical) face of $D$}. 

\begin{example}\label{example:Davis_complex}
Keeping the notations of Example~\ref{example:A2tilde}, the Davis complex $X$ of $(W,S)$ coincides with the geometric realisation of $\Sigma$ (i.e. the tesselated Euclidean plane pictured on Figure~\ref{figure:A2tilde}), together with the Euclidean metric $\dist$ (note that, in this example, all nonempty simplices of $\Sigma$ are spherical). The open cells of dimension $0,1,2$ are, respectively, the vertices, the edges without endpoints, and the open triangles, and these form a partition of $X$. 
\end{example}

\subsection{Actions on CAT(0)-spaces}\label{subsection:AOC}
Basics on $\CAT$ spaces can be found in \cite{BHCAT0}. Consider the $W$-action on $X$. For an element $w\in W$, we let
$$|w|:=\inf\{\dist(x,wx) \ | \ x\in X\}\in[0,+\infty)$$
denote its {\bf translation length}, and we set
$$\Min(w):=\{x\in X \ | \ \dist(x,wx)=|w|\}\subseteq X.$$
By a classical result of M.~Bridson (\cite{Bridson}), $\Min(w)$ is a nonempty closed convex subset of $X$ for all $w\in W$ (in particular, the infimum defining $|w|$ is always attained). More precisely, if $w$ has finite order, then $|w|=0$ and $\Min(w)$ is the fixed-point set of $w$ (see e.g. \cite[Theorem~11.23]{BrownAbr}). If $w$ has infinite order, then $|w|>0$ (otherwise, $w$ would fix a point $x$, and hence would belong to the finite parabolic subgroup $\Stab_W(x)$) and $\Min(w)$ is the union of all $w$-axes, where a {\bf $w$-axis} is a geodesic line stabilised by $w$ (on which $w$ then acts as a translation).

\begin{example}
Keeping the notations of Examples~\ref{example:A2tilde} and \ref{example:Davis_complex}, the element $v:=sut$ is a glide reflection, i.e. the composition of a translation with axis $L_v$ (depicted on Figure~\ref{figure:A2tilde}) with a flip around that axis. In particular, $L_v$ is the unique $v$-invariant line, and hence $\Min(v)=L_v$. On the other hand, $\Min(v^2)=X$, that is, $v^2$ is a translation across the whole plane.
\end{example}

\subsection{Walls}\label{subsection:W}
Given $x,y\in X$, we let $[x,y]$ denote the unique geodesic segment between $x$ and $y$. If $[x,y]$ intersects a wall $m$ of $X$ in at least two points, then it is entirely contained in $m$ (see \cite[Lemma~2.2.6]{Nos11}). Let $w\in W$ be of infinite order. A $w$-axis $L$ is {\bf transverse} to a wall $m$ if it intersects $m$ in a single point (in which case the two components of $L\setminus m$ lie on different sides of $m$, see \cite[Lemma~2.3.1]{Nos11}); in that case, any $w$-axis is transverse to $m$, and $m$ is called {\bf $w$-essential}. Note that, given any two points $x,y\in L$, there are only finitely many $w$-essential walls intersecting $[x,y]\subseteq L$. In particular, for any $x\in L$, there exists a (nonempty) open geodesic segment $\sigma\subseteq L$ containing $x$ in its closure and contained in some (open) cell $\supp(\sigma)$ (i.e. $\sigma$ does not intersect any $w$-essential wall).

\section{Tight conjugation}
We start by recalling the conjugation operations introduced in \cite{GP93} (see Definitions~\ref{definition:ra} and \ref{definition:strconj}), and then introduce a refinement of these notions, which we call ``tight conjugation'' (see Definition~\ref{definition:tightconj}). We also relate the operation of ``cyclic shift'' mentioned in the introduction (see Definition~\ref{definition:elrel}) to these operations.

\begin{definition}[{\cite{GP93}}]\label{definition:ra}
Let $w,w'\in W$ and $s\in S$. We write $w\ras{s} w'$ if $w'=sws$ and $\ell(w')\leq\ell(w)$. We write $w\ra w'$ if there is a sequence $w=w_0,w_1,\dots,w_n=w'$ of elements of $W$ such that, for each $i$, $w_{i-1}\ras{s_i}w_{i}$ for some $s_i\in S$.
\end{definition}

\begin{definition}\label{definition:elrel}
Let $w,w'\in W$. We say that $w'$ is a {\bf cyclic shift} of $w$ if there is a reduced decomposition $w=s_1\dots s_d$ ($s_i\in S$) of $w$ such that either $w'=s_2\dots s_ds_1$ or $w'=s_ds_1\dots s_{d-1}$ (that is, if a decomposition for $w'$ can be obtained from a reduced decomposition of $w$ by either moving the first letter at the end or the last letter at the beginning).
\end{definition}

\begin{definition}[{\cite{GP93}}]\label{definition:strconj}
Two elements $w,w'\in W$ are called {\bf elementarily strongly conjugate} if $\ell(w')=\ell(w)$, and there exists some $x\in W$ with $w'=x\inv wx$ such that either $\ell(x\inv w)=\ell(x)+\ell(w)$ or $\ell(wx)=\ell(w)+\ell(x)$; we then write $w\simx{x}w'$. We further call $w,w'\in W$ {\bf strongly conjugate} if there is a sequence $w=w_0,\dots,w_n=w'$ of elements of $W$ such that $w_{i-1}$ is elementarily strongly conjugate to $w_i$ for each $i$; we then write $w\sim w'$.
\end{definition}

\begin{definition}\label{definition:tightconj}
Two elements $w,w'\in W$ are called {\bf elementarily tightly conjugate} if $\ell(w)=\ell(w')$ and one of the following holds:
\begin{enumerate}
\item
there exists some $s\in S$ such that $w\ras{s}w'$.
\item
there exists some spherical subset $I\subseteq S$ such that $w\in N_W(W_I)$, and some $x\in W_I$ such that $w\simx{x}w'$.
\end{enumerate}
We further call $w,w'\in W$ {\bf tightly conjugate} if there is a sequence $w=w_0,\dots,w_n=w'$ of elements of $W$ such that $w_{i-1}$ is elementarily tightly conjugate to $w_i$ for each $i$; we then write $w\tc w'$.
\end{definition}

We now show that Definitions~\ref{definition:ra} and \ref{definition:elrel} yield equivalent concepts, and that ``tight conjugation'' is indeed a refinement of ``strong conjugation'' (but of course the two notions coincide if $W$ is finite).
\begin{lemma}\label{lemma:equivalence_relations}
Let $w,w'\in W$. Then the following assertions hold.
\begin{enumerate}
\item
$w\ra w'$ $\iff$ $w'$ can be obtained from $w$ by a sequence of cyclic shifts.
\item
$w\tc w'$ $\implies$ $w\sim w'$.
\item
If $\ell(w)=\ell(w')$, then $w\ra w'$ $\implies$ $w\tc w'$.
\end{enumerate}
\end{lemma}
\begin{proof}
(1) For the forward implication, it is sufficient to show that if $w\ras{s} sws$ for some $s\in S$, then either $sws=w$ (no cyclic shift made) or $sws$ is a cyclic shift of $w$. We may thus assume that $sws\neq w$ and $\ell(sws)\leq\ell(w)$. Then either $\ell(sw)<\ell(w)$ or $\ell(ws)<\ell(w)$ (see the condition (F) in \cite[page 79]{BrownAbr}), and hence $sws$ is a cyclic shift of $w$ by the exchange condition (\cite[Condition (E) page 79]{BrownAbr}). The converse is clear.

(2) It is sufficient to show that if $w\ras{s}sws$ with $\ell(w)=\ell(sws)$ for some $s\in S$, then either $w=sws$ (so that $w\simx{1} sws$) or $w\simx{s} sws$. Assume that $\ell(sws)=\ell(w)$, and that $\ell(sw)<\ell(w)$ and $\ell(ws)<\ell(w)$ for some $s\in S$, and let us show that $sws= w$. As $\ell(sw)<\ell(w)$, the exchange condition implies that $w$ has a reduced decomposition $w=s_1\dots s_n$ with $s_1=s$. Similarly, as $\ell(ws)<\ell(w)$, the exchange condition implies that there is some $i\in\{1,\dots,n\}$ such that $w=s_1\dots \widehat{s_i}\dots s_ns$, where $\widehat{s_i}$ indicates the omission of $s_i$. If $i\neq 1$, then $\ell(sws)=\ell(s_2\dots \widehat{s_i}\dots s_n)=\ell(w)-2$, a contradiction. Thus $i=1$ and $sws=ss_2\dots s_n=w$, as desired.

(3) This holds by definition of tight conjugation.
\end{proof}

\begin{definition}
For $w,w'\in W$, we write $w\ra\tc w'$ if there is some $w''\in W$ such that $w\ra w''$ and $w''\tc w'$.
\end{definition}

\section{The complex \texorpdfstring{$\CMin(w)$}{CombiMin(w)}}\label{section:TCC}

In this section, we establish some basic properties of the combinatorial analogue $\CMin(w)$ of $\Min(w)$ for an element $w\in W$, and show how it is related to the conjugation operation $\ra$ from Definition~\ref{definition:ra}.

\begin{definition}
For $w\in W$, set $$\CMin(w):=\{D\in\Ch(X) \ | \ \textrm{$\dist_{\Ch}(D,wD)$ is minimal}\}.$$
Alternatively, $\CMin(w)$ is the set of chambers $D=vC_0$ ($v\in W$) such that $v\inv wv$ is of minimal length in the conjugacy class of $w$. In other words, $\CMin(w)$ coincides with the inverse image under the map $$\pi_w\co\Ch(X)\to W: vC_0\mapsto v\inv wv$$ 
of the set of conjugates of $w$ of minimal length.
\end{definition}

\begin{definition}
Let $w\in W$. A chamber subcomplex $A$ of $X$ is called {\bf $w$-convex} if $\Gamma(D,w^{\varepsilon}D)\subseteq A$ for any chamber $D$ of $A$ and any $\varepsilon\in\{\pm 1\}$.
\end{definition}

\begin{lemma}\label{lemma:CMin_wconvex}
Let $w\in W$. Then $\CMin(w)$ is $w$-convex.
\end{lemma}
\begin{proof}
Let $\varepsilon\in\{\pm 1\}$ and $D\in\CMin(w)$. Let $E\in\Gamma(D,w^{\varepsilon}D)$, and let $\Gamma_1$ (resp. $\Gamma_2$) be a minimal gallery from $D$ to $E$ (resp. from $E$ to $w^{\varepsilon}D$), so that $\ell(\Gamma_1)+\ell(\Gamma_2)=\dc(D,w^{\varepsilon}D)$. Then the concatenation of $\Gamma_2$ with $w^{\varepsilon}\Gamma_1$ is a gallery from $E$ to $w^{\varepsilon}E$, and hence $\dc(E,w^{\varepsilon}E)\leq\dc(D,w^{\varepsilon}D)$, yielding the claim.
\end{proof}

\begin{lemma}\label{lemma:Cmin_straight_R}
Let $w\in W$ be straight, and let $R$ be a spherical residue with $\Stab_W(R)=\Stab_W(wR)$.  Let $C,D\in R\cap\CMin(w)$. Then $\pi_w(C)= \pi_w(D)$.
\end{lemma}
\begin{proof}
Let $u,v\in W$ be such that $C=uC_0$ and $D=vC_0$, and let us show that $u\inv wu= v\inv wv$.
Note that $w$ is of minimal length in its conjugacy class by Lemma~\ref{lemma:straight_lemma_4.1}. Hence $w_0:=u\inv w u$ is straight by Lemma~\ref{lemma:straight_lemma_4.2} (because $uC_0\in\CMin(w)$). On the other hand, writing $\Stab_W(R)=uW_Iu\inv$ for some spherical subset $I\subseteq S$, the hypotheses imply that $w_0\in N_W(W_I)$. From Lemma~\ref{lemma:straight_lemma_4.1}, we then deduce that $w_0\in N_I$. In particular, $w_0C_0=\proj_{w_0R_0}(C_0)$ and hence $\dc(C_0,w_0C_0)=\dc(E,\proj_{w_0R_0}(E))$ for all $E\in R_0:=u\inv R$ (see Example~\ref{example:RwR}). Finally, since $E:=u\inv vC_0\in R_0\cap\CMin(w_0)$ by assumption, we have $\dc(E,w_0E)=\dc(C_0,w_0C_0)$ and hence $w_0E=\proj_{w_0R_0}(E)$.
Since $v\inv uw_0R_0=w_0R_0$ (because $u\inv v\in W_I=\Stab_W(R_0)$ and $w_0\in N_W(W_I)$), we conclude that
$$w_0E=\proj_{w_0R_0}(u\inv vC_0)=u\inv v\proj_{v\inv uw_0 R_0}(C_0)=u\inv v\proj_{w_0 R_0}(C_0)=u\inv vw_0C_0,$$
that is, $w_0u\inv v=u\inv vw_0$, or else $v\inv wv=v\inv uw_0u\inv v=w_0=u\inv wu$, as desired.
\end{proof}

\begin{lemma}\label{lemma:connection_gallery_Cw}
Let $w\in W$. Let $C,D\in\Ch(X)$ be two chambers connected by a gallery $\Gamma\subseteq\CMin(w)$. Then $\pi_w(C)\ra \pi_w(D)$.
\end{lemma}
\begin{proof}
Let $u,v\in W$ be such that $C=uC_0$ and $D=vC_0$, and let us show that $u\inv wu\ra v\inv wv$.
Reasoning inductively on $\ell(\Gamma)$, we may assume that $uC_0,vC_0$ are adjacent, say $v=us$ for some $s\in S$. As $uC_0,vC_0\in\CMin(w)$, we have $\ell(u\inv wu)=\ell(v\inv wv)$, and hence $u\inv wu\ras{s}v\inv wv$, as desired.
\end{proof}

\begin{remark}\label{remark:CMin_not_convex}
Let $w\in W$ be of minimal length in its conjugacy class (i.e. $C_0\in\CMin(w)$). Lemma~\ref{lemma:connection_gallery_Cw} implies that if $\CMin(w)$ is gallery-connected, then every conjugate $w'$ of $w$ with $\ell(w')=\ell(w)$ can be obtained from $w$ by a sequence of cyclic shifts (in particular, Theorem~\ref{thmintro:mainthm}(2) holds for $w$). The necessity of introducing ``tight conjugations'' as well comes from the fact that $\CMin(w)$ need not be gallery-connected, as for instance illustrated by the Coxeter group $W=\langle s,t \ | \ s^2=t^2=(st)^3=1\rangle$ of type $A_2$, with $w=s$.
\end{remark}

\section{The complex \texorpdfstring{$\CCC^w$}{Cw}}\label{section:TCCw}
In this section, we define for each $w\in W$ a chamber subcomplex $\CCC^w$ of $X$ such that for any chamber $D=vC_0$ of $\CCC^w$ ($v\in W$), the conjugate $\pi_w(D)=v\inv wv$ of $w$ can be obtained from $w$ through a sequence of cyclic shifts and tight conjugations.

\begin{definition}
Let $w\in W$. Consider the following conditions, which a chamber subcomplex $A$ of $X$ may or may not satisfy.
\begin{enumerate}
\item[(CM0)] $C_0\in\Ch(A)$.
\item[(CM1)] If $C\in\Ch(A)$ and $D\in\Gamma(C,w^{\varepsilon}C)$ for some $\varepsilon\in\{\pm 1\}$ and $\dc(C,D)=1$, then $D\in\Ch(A)$.
\item[(CM2)] If $R$ is a spherical residue such that $\Stab_W(R)=\Stab_W(wR)$ and $R\cap \Ch(A)\neq\varnothing$, then $R\cap\CMin(w)\subseteq A$.
\end{enumerate}
We let $\CCC^w$ (resp. $\CCC^w_1$) denote the smallest chamber subcomplex of $X$ satisfying (CM0), (CM1) and (CM2) (resp. satisfying (CM0) and (CM1)).
\end{definition}

\begin{lemma}\label{lemma:Cw_in_Cminw}
Let $w\in W$ be of minimal length in its conjugacy class. Then $\CCC^w\subseteq\CMin(w)$.
\end{lemma}
\begin{proof}
We have to check that $\CMin(w)$ satisfies (CM0), (CM1) and (CM2). But (CM0) holds by assumption, (CM1) by Lemma~\ref{lemma:CMin_wconvex}, and (CM2) is clear.
\end{proof}

\begin{lemma}\label{lemma:ra}
Let $w\in W$, and let $C,D\in\Ch(X)$ be such that $D\in\Gamma(C,w^{\varepsilon}C)$ for some $\varepsilon\in\{\pm 1\}$ and $\dc(C,D)=1$. Then $\pi_w(C)\ra \pi_w(D)$. If, moreover, $C\in\CMin(w)$, then $\pi_w(C)\tc \pi_w(D)$.
\end{lemma}
\begin{proof}
Let $u\in W$ be such that $C=uC_0$. Let $\Gamma$ be a minimal gallery from $uC_0$ to $w^{\varepsilon}uC_0$ containing $D$, and let $(s_{1},\dots,s_{d})$ be its type, so that $D=us_1C_0$ and $u\inv w^{\varepsilon}u=s_{1}\dots s_{d}$. Then $\pi_w(C)=u\inv wu$ is either $s_1\dots s_d$ or $s_d\dots s_1$, and hence $\pi_w(C)\ras{s_1}\pi_w(D)=s_1\pi_w(C)s_1$, yielding the first claim. The second claim then follows from Lemma~\ref{lemma:equivalence_relations}(3).
\end{proof}

The following lemma is an adaptation of \cite[Proposition~3.4]{straight}.
\begin{lemma}\label{lemma:axisthroughC0}
Let $w\in W$ be of infinite order. Then $\CCC^w_1\cap\Min(w)$ is nonempty. In particular, there exists $w'\in W$ such that $w\ra w'$ and $\Min(w')\cap C_0\neq\varnothing$. 
\end{lemma}
\begin{proof}
For each $u\in W$, consider the continuous function $f_u\co X\to\RR:x\mapsto\dist(x,ux)$. Note that if $vC_0\in\Ch(\CCC^w_1)$, then $w\ra v\inv wv$ by Lemma~\ref{lemma:ra}, and hence $\ell(v\inv wv)\leq\ell(w)$. In particular, the set $\{v\inv wv \ | \ vC_0\in\Ch(\CCC^w_1)\}$ is finite. Since $C_0$ is compact, we deduce that the set $\{f_{v\inv wv}(y) \ | \ y\in C_0, \ vC_0\in\Ch(\CCC_1^w)\}$ contains its infimum $a$. Let $y\in C_0$ and $v\in W$ with $vC_0\in\Ch(\CCC_1^w)$ be such that $f_{v\inv wv}(y)=a$. Then $f_w$ attains its infimum over $\CCC^w_1$ at $x:=vy\in vC_0$, for if $z\in \CCC_1^w$, then choosing $uC_0\in\Ch(\CCC_1^w)$ ($u\in W$) containing $z$, we have 
$$f_w(z)=\dist(z,wz)=\dist(u\inv z,u\inv wu.u\inv z)=f_{u\inv wu}(u\inv z)\geq a=f_{v\inv wv}(v\inv x)=f_w(x).$$

We claim that $\dist(z,wz)=\dist(x,wx)$ for some $z\in [x,wx]\setminus\{x,wx\}$, so that $x\in\Min(w)$ (see e.g. \cite[Proposition~II.1.4(2)]{BHCAT0}), yielding the first assertion of the lemma. Indeed, let $D\in\Ch(\CCC^w_1)$ with $x\in D$ and such that $\dc(D,wD)$ is minimal for these properties. By \cite[Lemma~3.1]{straight}, there exists a minimal gallery $\Gamma=(D=D_0,D_1,\dots,D_k=wD)$ from $D$ to $wD$ containing the geodesic segment $[x,wx]$. 

Let $\varepsilon>0$ and $i\geq 0$ be such that $[x,wx]\cap B(x,\varepsilon)\subseteq D_i$ (where $B(x,\varepsilon):=\{y\in X \ | \ \dist(x,y)\leq\varepsilon\}$). In particular, $D_0,\dots,D_i$ contain $x$. Moreover, $D_0,\dots,D_i\in\Ch(\CCC^w_1)$: indeed, $D_0\in\Ch(\CCC^w_1)$ by assumption, and if $D_j\in\Ch(\CCC^w_1)$ for some $j\in\{0,\dots,i-1\}$, then $\dc(D_j,wD_j)\geq\dc(D,wD)=k$ by the minimality assumption on $D$, so that $(D_j,D_{j+1},\dots,D_k,wD_1,\dots,wD_{j})$ is a minimal gallery from $D_j$ to $wD_j$, and hence $D_{j+1}\in\Ch(\CCC^w_1)$ by (CM1). 

Let now $z\in D_i\cap [x,wx]$ with $z\notin \{x,wx\}$. Then $f_w(z)\geq f_w(x)$, that is, $\dist(z,wz)\geq\dist(x,wx)$, and since $\dist(z,wz)\leq \dist(z,wx)+\dist(wx,wz)=\dist(x,wx)$, the claim follows.

For the second assertion of the lemma, let $x\in \CCC^w_1\cap\Min(w)$, and let $u\in W$ be such that $uC_0\in\Ch(\CCC^w_1)$ and $x\in uC_0$. As noticed at the beginning of the proof, we have $w\ra u\inv wu$. We may thus choose $w':=u\inv wu$, as $\Min(w')=u\inv\Min(w)$ and hence $u\inv x\in \Min(w')\cap C_0$.
\end{proof}

\begin{lemma}\label{lemma:approx}
Let $w\in W$ and let $C\in\Ch(X)$ be such that $C\in R$ for some spherical residue $R$ with $\Stab_W(R)=\Stab_W(wR)$. Let $D\in\Ch(X)$ be such that $D\in R\cap\CMin(w)$. Then $\pi_w(C)\ra\tc \pi_w(D)$. If, moreover, $C\in\CMin(w)$, then $\pi_w(C)\tc \pi_w(D)$.
\end{lemma}
\begin{proof}
Let $u,v\in W$ be such that $C=uC_0$ and $D=vC_0$.
As $u\inv R$ is a standard spherical residue, there is some spherical subset $I\subseteq S$ such that $\Stab_W(u\inv R)=W_I$. By assumption, $u\inv w u$ normalises $W_I$, and hence there exist by Lemma~\ref{lemma:Lustzig} some $w_I\in W_I$ and $n_I\in N_I$ such that $u\inv w u=n_Iw_I$. Moreover, as $vC_0\in R$, there is some $x\in W_I$ such that $v=ux$. Let $\delta\co W_I\to W_I$ denote the diagram automorphism of $W_I$ defined by $\delta(z):=n_I\inv zn_I$. Then
$$v\inv wv=x\inv n_Iw_Ix=n_I\cdot \delta(x)\inv w_Ix.$$

Note that the element $\delta(x)\inv w_Ix$ is of minimal length in its $\delta$-twisted conjugacy class $\OOO_{\delta}(w_I):=\{\delta(z)\inv w_Iz \ | \ z\in W_I\}$: otherwise, we find some $z\in W_I$ such that $\ell(\delta(xz)\inv w_Ixz)<\ell(\delta(x)\inv w_Ix)$. We then deduce from (\ref{eqn:Lustzig}) in Lemma~\ref{lemma:Lustzig} that $$\ell((vz)\inv wvz)=\ell(n_I\cdot\delta(xz)\inv w_Ixz)<\ell(n_I\cdot \delta(x)\inv w_Ix)=\ell(v\inv wv),$$
contradicting our assumption that $v\inv wv$ is of minimal length in its conjugacy class (i.e. $vC_0\in\CMin(w)$). 

By \cite[Theorem~3.1]{HN12} applied to the Coxeter system $(W_I,I)$ and to the automorphism $\delta$ of $W_I$, we can find some $w'_I\in W_I$ of minimal length in $\OOO_{\delta}(w_I)$ such that $w_I\ra_{\delta}w'_I$ and $w'_I\sim_{\delta}\delta(x)\inv w_Ix$, where $\ra_{\delta}$ and $\sim_{\delta}$ are the analogues of $\ra$ and $\sim$ in $W_I$ for $\delta$-twisted conjugacy classes (i.e. one can transform $w_I$ into $w_I'$ by a sequence of elementary operations of the form $z\mapsto \delta(s)\inv zs$ ($z\in W_I$ and $s\in I$) where $\ell(\delta(s)\inv zs)\leq \ell(z)$, and one can transform $w_I'$ into $\delta(x)\inv w_Ix$ by a sequence of elementary operations of the form $z\mapsto \delta(y)\inv zy$ ($z,y\in W_I$) where $\ell(\delta(y)\inv zy)=\ell(z)$ and either $\ell(\delta(y)\inv z)=\ell(y)+\ell(z)$ or $\ell(zy)=\ell(z)+\ell(y)$).

Using again (\ref{eqn:Lustzig}) and the fact that $n_I\cdot \delta(y)\inv zy=y\inv n_Izy$ for all $y,z\in W_I$, we deduce that $$n_I\cdot w_I\ra n_I\cdot w'_I\tc n_I\cdot \delta(x)\inv w_Ix,$$ that is, $u\inv wu=n_Iw_I\ra\tc n_I\cdot \delta(x)\inv w_Ix=x\inv n_Iw_Ix=v\inv wv$, proving the first claim. The second claim then follows from Lemma~\ref{lemma:equivalence_relations}(3).
\end{proof}

\begin{prop}\label{prop:translation_geom_combi}
Let $w\in W$. If $C\in\Ch(\CCC^w)\cap\CMin(w)$, then $w\ra\tc \pi_w(C)$.
\end{prop}
\begin{proof}
Note that $w=\pi_w(C_0)$. By definition of $\CCC^w$, the chamber $C$ can be obtained from the chamber $C_0$ after performing a finite sequence of steps of one of the following two types:
\begin{enumerate}
\item[(I)]
going from a chamber $C\in\Ch(\CCC^w)$ to a chamber $D\in\Gamma(C,w^{\varepsilon}C)$ for some $\varepsilon\in\{\pm 1\}$, and with $\dc(C,D)=1$.
\item[(II)]
going from a chamber $C\in R\cap\Ch(\CCC^w)$ for some spherical residue $R$ with $\Stab_W(R)=\Stab_W(wR)$ to a chamber $D\in R\cap\CMin(w)$.
\end{enumerate}
Hence the proposition follows from a straightforward induction on the number of steps of type (I) and (II) needed to go from $C_0$ to $C$, by using Lemmas~\ref{lemma:ra} and \ref{lemma:approx}.
\end{proof}

\begin{remark}\label{remark:CMin_in_Cw}
Let $w\in W$. If $\CMin(w)\subseteq \CCC^w$, then Proposition~\ref{prop:translation_geom_combi} implies that $w\ra\tc u$ for any $u$ of minimal length in the conjugacy class of $w$, thus proving Theorem~\ref{thmintro:mainthm}(1,2) in that case. This idea will be implemented in the next section to complete the proof of Theorem~\ref{thmintro:mainthm}.
\end{remark}

\section{The conjugacy problem in \texorpdfstring{$(W,S)$}{(W,S)}}\label{section:TCPIWS}

This section is devoted to the proof of Theorem~\ref{thmintro:mainthm}.

\begin{remark}\label{remark:finite_rank_reduction}
Note that, in order to prove Theorem~\ref{thmintro:mainthm}, there is no loss of generality in assuming that $(W,S)$ has finite rank (i.e. that $S$ is finite), justifying our standing assumption from the beginning of \S\ref{subsection:BD}. Indeed, if $w,w'\in W$ are conjugate, say $w'=v\inv wv$ for some $v\in W$, then there is some finite subset $J\subseteq S$ such that $w,w',v\in W_J$, and it is thus sufficient to show that $w$ and $w'$ are related by a suitable sequence of elementary operations inside the finite rank Coxeter system $(W_J,J)$.
\end{remark}

\begin{prop}\label{prop:projection}
Let $w\in W$ be of infinite order. Let $L$ be a $w$-axis, and let $\sigma\subseteq L$ be a nonempty open geodesic segment that is contained in some open cell $\supp(\sigma)$. Let $R$ be the spherical residue corresponding to $\supp(\sigma)$. Let $D$ be a chamber. Then  $\dc(C,wC)\leq\dc(D,wD)$, where $C:=\proj_R(D)$. 

In particular, if $D\in\CMin(w)$, then $\proj_R(D)\in\CMin(w)$.
\end{prop}
\begin{proof}
Without loss of generality, we may assume that $C=C_0$: indeed, write $C=vC_0$ for some $v\in W$. Then $L':=v\inv L$ is an axis for $w':=v\inv w v$ containing the nonempty open geodesic segment $\sigma':=v\inv\sigma$, and $R':=v\inv R$ is the spherical residue corresponding to the cell $\supp(\sigma'):=v\inv\supp(\sigma)$ supporting $\sigma'$. Moreover, setting $D':=v\inv D$, we have $\proj_{R'}(D')=v\inv \proj_R(D)=v\inv C=C_0$. As $\dc(C,wC)=\dc(C_0,w'C_0)$ and $\dc(D,wD)=\dc(D',w'D')$, the claim follows.

Assume thus that $C=C_0$; in particular, $R$ is standard. Let $I\subseteq S$ be such that $\Stab_W(R)=W_I$. Let us show that  
\begin{equation}\label{eqn:TP}
\ell(w)=\dc(C_0,wC_0)\leq\dc(D,wD).
\end{equation}
 Note that the walls of $R$ coincide with the walls containing $\sigma$, or else with the walls containing $L$. In particular, $w$ stabilises this set of walls, so that the residues $R$ and $wR$ are parallel, and $w\in N_W(W_I)$. Write $w=n_Iw_I$ for some $w_I\in W_I$ and $n_I\in N_I$, so that $\ell(w)=\ell(n_I)+\ell(w_I)$ (see Lemma~\ref{lemma:Lustzig}). Thus the chambers $C_0$ and $n_IC_0$ lie on the same side of any wall of $R$.

Let $n\in\NN^*$ be such that $w^n=n_I^n$ (i.e. for each $r\in\NN$, there is some $w_r\in W_I$ such that $w^r=n_I^rw_r$; since $W_I$ is finite, we find some $r,s\in\NN^*$ with $r<s$ such that $w_r=w_s$, and one can take $n:=s-r$). Let $\Gamma$ be a gallery from $D$ to $w^nD$ obtained by concatenating minimal galleries $\Gamma_i$ from $w^{i-1}D$ to $w^iD$ for $i=1,\dots,n$. Thus $\ell(\Gamma)=n\dc(D,wD)$. 

Since $D,C_0,n_IC_0$ lie on the same side of any wall of $R$ (equivalently, of $wR$), we have $\proj_{wR}(D)=n_IC_0$, and hence $$\proj_R(w\inv D)=w\inv\proj_{wR}(D)=w\inv n_IC_0=w_I\inv C_0.$$
In particular, for each $i\in\{1,\dots,n\}$, the number of walls of $R$ (or equivalently, of $w^{i}R$) crossed by $\Gamma_i$ is
$$\dc(\proj_{w^{i}R}(w^{i-1}D),\proj_{w^iR}(w^iD))=\dc(\proj_{R}(w\inv D),\proj_{R}(D))=\dc(C_0,w_IC_0).$$
As $\ell(\Gamma)$ is also the number of times $\Gamma$ crosses a wall, and as $D$ and $w^nD=n_I^nD$ lie on the same side of any wall of $R$, we deduce that
\begin{equation*}
n\dc(D,wD)=\ell(\Gamma)\geq \dc(D,w^nD)+n\dc(C_0,w_IC_0).
\end{equation*}
In particular,
\begin{equation}\label{eqn:straight_cons}
\ell(w)=\ell(n_I)+\dc(C_0,w_IC_0)\leq\ell(n_I)+\dc(D,wD)-\frac{\dc(D,w^nD)}{n}.
\end{equation}

Finally, note that $L$ is also an $n_I$-axis, as $n_Ix=n_Iw_Ix=wx\in L$ for any $x\in L$. As $C_0$ and $n_IC_0$ are not separated by any wall containing $L$ (that is, by any wall of $R$), it follows from \cite[Lemma~4.3]{straight} that $n_I$ is straight. Hence $$n\ell(n_I)=\dc(C_0,n_I^nC_0)=\dc(C_0,w^nC_0).$$ Moreover, $n_I^n=w^n$ is straight as well, hence of minimal length in its conjugacy class by Lemma~\ref{lemma:straight_lemma_4.1}. In particular, $\dc(C_0,w^nC_0)\leq\dc(D,w^nD)$. Therefore, $\dc(D,w^nD)\geq n\ell(n_I)$, and (\ref{eqn:TP}) follows from (\ref{eqn:straight_cons}).
\end{proof}

\begin{prop}\label{prop:rasim}
Let $w\in W$ be of infinite order, and let $u$ be of minimal length in the conjugacy class of $w$. Then $w\ra\tc u$. If, moreover, $w$ is straight, then $w\ra u$. 
\end{prop}
\begin{proof}
By Lemma~\ref{lemma:axisthroughC0}, we find some $w_1,u_1\in W$ with $w\to w_1$ and $u\to u_1$, such that there exist some $x_w\in\Min(w_1)\cap C_0$ and some $x_u\in\Min(u_1)\cap C_0$. Note that $u_1$ is still of minimal length in its conjugacy class and $u_1\to u$; similarly, if $w$ is straight, then $w_1$ is still straight by Lemma~\ref{lemma:straight_lemma_4.2}. In view of Lemma~\ref{lemma:equivalence_relations}(3), there is thus no loss of generality in assuming that $w=w_1$ and $u=u_1$.

Let $v\in W$ be such that $u=v\inv wv$. In particular, $vC_0\in\CMin(w)$. Moreover, $Z:=[x_w,vx_u]\subseteq\Min(w)$ as $\Min(w)$ is convex. Let $\Gamma_Z(C_0,vC_0)$ be the set of chambers of $\Gamma(C_0,vC_0)$ intersecting $Z$ nontrivially (note that there always exists a minimal gallery from $C_0$ to $vC_0$ containing $Z$, see \cite[Lemma~3.1]{straight}).

\medskip
\noindent
{\bf Claim:} \emph{Let $D\in\Gamma_Z(C_0,vC_0)\cap\Ch(\CCC^w)$ with $D\neq vC_0$. Then there exists some $E\in\Gamma_Z(C_0,vC_0)\cap\Ch(\CCC^w)$ with $\dc(E,vC_0)<\dc(D,vC_0)$. If, moreover, $w$ is straight, then $\pi_w(D)\to\pi_w(E)$.}

\medskip

Indeed, let $x_D\in D\cap Z$, and let $D=D_0,D_1,\dots,D_k=vC_0$ be a minimal gallery from $D$ to $vC_0$ containing $[x_D,vx_u]$ ($k\geq 1$). Let $x\in (D_0\cap D_1)\cap [x_D,vx_u]$. Let also $L$ be the $w$-axis through $x$, and let $\sigma\subseteq L$ be a nonempty open geodesic segment containing $x$ in its closure and contained in some (open) cell $\supp(\sigma)$. Let $R_x$ (resp. $R_{\sigma}$) be the spherical residue consisting of all chambers containing $x$ (resp. $\sigma$). In particular, $R_{\sigma}\subseteq R_x$ and $D\in R_x$. Let $D':=\proj_{R_{\sigma}}(D)$, $E=\proj_{R_x}(vC_0)$ and $E':=\proj_{R_{\sigma}}(vC_0)=\proj_{R_{\sigma}}(E)$ (see Figure~\ref{figure:Prop63}). Note that $E\in\Gamma_Z(C_0,vC_0)$, because $\dc(vC_0,D)=\dc(vC_0,E)+\dc(E,D)$ by the gate property. Moreover, $D\neq E$, for otherwise $\dc(vC_0,D_1)=\dc(vC_0,D)+1$, a contradiction. In particular, $\dc(E,vC_0)<\dc(D,vC_0)$.

Let now $\Gamma_D=(D=D'_0,D'_1,\dots,D'_l=D')$ be a minimal gallery from $D$ to $D'$. We claim that $\Gamma_D\subseteq\CCC^w$. Indeed, assume for a contradiction that there is some $i\in\{1,\dots,l\}$ such that $D'_{i-1}\in\Ch(\CCC^w)$ but $D'_i\notin\Ch(\CCC^w)$. Let $m$ be the wall separating $D'_{i-1}$ from $D'_i$. Note that $x\in D'_{i-1}\cap D'_i$. Hence if $L$ is transverse to $m$, then we find some $\varepsilon\in\{\pm 1\}$ such that $D'_{i-1}$ and $w^{\varepsilon}D'_{i-1}$ lie on different sides of $m$. Hence in that case, $D'_i\in\Gamma(D'_{i-1},w^{\varepsilon}D'_{i-1})$, so that $D'_i\in\Ch(\CCC^w)$ by (CM1), a contradiction (this is illustrated for $i=2$ on Figure~\ref{figure:Prop63}). Thus $m$ contains $L$. In particular, $m$ contains $\sigma$, and hence $m$ is a wall of $R_{\sigma}$. But by definition of $D'$, the minimal gallery $\Gamma_D$ does not cross any wall of $R_{\sigma}$, a contradiction. This proves that $\Gamma_D\subseteq\CCC^w$, and hence in particular that $D'\in\Ch(\CCC^w)$.

We next claim that $E'\in\Ch(\CCC^w)$. Indeed, note that the walls of $R_{\sigma}$ coincide with the walls containing $L$. In particular, $w$ stabilises this set of walls. In other words, $\Stab_W(R_{\sigma})=\Stab_W(wR_{\sigma})$. On the other hand, since $vC_0\in\CMin(w)$ by assumption, Proposition~\ref{prop:projection} implies that $E'\in\CMin(w)$. Hence $E'\in\Ch(\CCC^w)$ by (CM2).

Finally, let $\Gamma_E$ be a minimal gallery from $E'$ to $E$. Then $\Gamma_E\subseteq\CCC^w$, for exactly the same reasons that $\Gamma_D\subseteq\CCC^w$. In particular, $E\in\Ch(\CCC^w)$.

If, moreover, $w$ is straight (in particular, $w$ is of minimal length in its conjugacy class by Lemma~\ref{lemma:straight_lemma_4.1}), then $\Gamma_D,\Gamma_E\subseteq\CCC^w\subseteq\CMin(w)$ by Lemma~\ref{lemma:Cw_in_Cminw}, and hence $\pi_w(D)\to\pi_w(D')$ and $\pi_w(E')\to\pi_w(E)$ by Lemma~\ref{lemma:connection_gallery_Cw}. As $\pi_w(D')=\pi_w(E')$ by Lemma~\ref{lemma:Cmin_straight_R}, this proves the claim. 

\medskip

The claim readily implies that $vC_0\in\Ch(\CCC^w)$, so that $w\ra\tc u$ by Proposition~\ref{prop:translation_geom_combi}. If, moreover, $w$ is straight, then the claim implies that $w=\pi_w(C_0)\to\pi_w(vC_0)=u$, as desired.
\end{proof}

\begin{figure}

  \centering
  \includegraphics[trim = 36mm 191mm 91mm 28mm, clip, width=9cm]{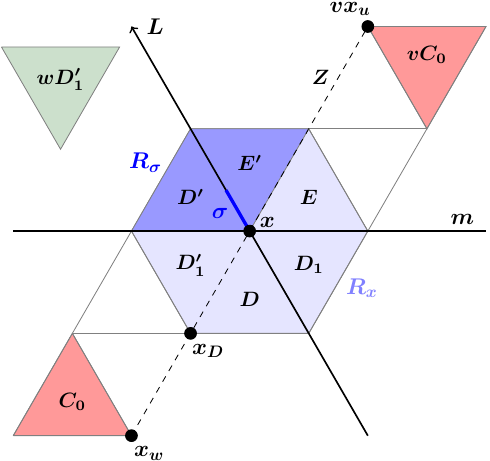}
  \captionof{figure}{Proof of Proposition~\ref{prop:rasim}}
  \label{figure:Prop63}

\end{figure}

\begin{lemma}\label{lemma:sph_conj}
Let $I,J\subseteq S$ be such that $W_I$ and $W_J$ are conjugate. Then there is some $x\in W$ with $x\inv Ix=J$ such that $w\tc x\inv wx$ for all $w\in W_I$.
\end{lemma}
\begin{proof}
By \cite[Proposition~3.1.6]{Kra09}, there exists some $x\in W$ such that $x\inv \Pi_I=\Pi_J$. Hence by \cite[Theorem~3.1.3]{Kra09} (see also \cite[Proposition~5.5]{Deo82}), we find a sequence $I=I_0,I_1,\dots,I_{k+1}=J$ of subsets of $S$ and elements $s_i\in S\setminus I_i$ ($i=0,\dots,k$) such that the following hold for each $i\in\{0,\dots,k\}$:
\begin{enumerate}
\item
In the Coxeter diagram of $(W,S)$, the connected component $K_i$ of $I_i\cup\{s_i\}$ containing $s_i$ is spherical. We set $x_i:=w_{K_i\setminus\{s_i\}}w_{K_i}$, where for a spherical subset $T\subseteq S$ we denote by $w_T$ the longest element of $W_T$.
\item
$x_i\inv\Pi_{I_i}=\Pi_{I_{i+1}}$.
\item
$I_{i+1}=(I_i\cup\{s_i\})\setminus\{t_i\}$ for some $t_i\in K_i$.
\item
$x=x_0\dots x_k$ and $\ell(x)=\sum_{j=0}^k\ell(x_j)$.
\end{enumerate}
(Note that $x_i$ is denoted $\nu(I_i,s_i)$ in \cite[Theorem~3.1.3]{Kra09}.)
Let $w\in W_I$. Reasoning inductively on $k$, it is sufficient to show that if $J=(I\cup\{s\})\setminus\{t\}$ for some $s\in S\setminus I$ and some $t\in K$, where $K$ is the connected component of $I\cup\{s\}$ containing $s$ ($K$ is spherical), then $w\simx{x}x\inv wx$, where $x:=w_{K\setminus\{s\}}w_{K}\in W_K$ (note that $w$ normalises $W_K$).

Set $I':=I\setminus K$, so that $I'$ and $K$ are not connected in $I'\cup K=I\cup\{s\}$. Write $w=w(I')\cdot w(K^s)$ with $w(I')\in W_{I'}$ and $w(K^s)\in W_{K^s}$, where we set for short $K^s:=K\setminus\{s\}=I\setminus I'$. We have to show that $\ell(x\inv w)=\ell(x)+\ell(w)$. Recall from \cite[Proposition~1.77]{BrownAbr} that for any spherical subset $T\subseteq S$ and any $v\in W_T$, we have $w_T=w_T\inv$ and $\ell(w_Tv)=\ell(vw_T)=\ell(w_T)-\ell(v)$. Hence,
\begin{align*}
\ell(x\inv w)&=\ell(w_K\inv \cdot w_{K^s}\inv w(K^s))+\ell(w(I'))=\ell(w_K)-\ell(w_{K^s}\inv w(K^s))+\ell(w(I'))\\
&=\ell(w_K)-\ell(w_{K^s})+\ell(w(K^s))+\ell(w(I'))=\ell(w_{K^s}w_K)+\ell(w(K^s)w(I'))\\
&=\ell(x)+\ell(w),
\end{align*}
as desired.
\end{proof}

\begin{prop}\label{prop:main_result_finite}
Let $w\in W$ be of finite order, and let $u$ be of minimal length in the conjugacy class of $w$. Then $w\ra \tc u$.
\end{prop}
\begin{proof}
By \cite[Corollary~C]{straight}, there is some $w_1\in W$ of minimal length in its conjugacy class such that $w\ra w_1$, and we may thus assume that $w$ is of minimal length in its conjugacy class. In particular, by \cite[Proposition~4.2]{CF10}, $w$ has standard parabolic closure $\Pc(w)=W_I$ ($I\subseteq S$), while $u$ has standard parabolic closure $\Pc(u)=W_J$ ($J\subseteq S$). As $W_I$ and $W_J$ are conjugate, we find some $w'\in W_J$ such that $w\tc w'$ by Lemma~\ref{lemma:sph_conj}. On the other hand, by \cite[Theorem~3.1(2)]{HN12} applied in the finite Coxeter group $W_J$, we have $w'\sim u$ (and hence $w'\tc u$). Thus $w\tc w'\tc u$, as desired.
\end{proof}

Here is a reformulation of Theorem~\ref{thmintro:mainthm}.

\begin{theorem}
Let $w\in W$, and let $u$ be of minimal length in the conjugacy class of $w$. Then $w\ra\tc u$. If, moreover, $w$ is straight, then $w\ra u$.
\end{theorem}
\begin{proof}
This sums up Propositions~\ref{prop:rasim} and \ref{prop:main_result_finite}.
\end{proof}

\section{Comparison of \texorpdfstring{$\Min(w)$ and $\CMin(w)$}{Min(w) and CombiMin(w)}}\label{section:COMAC}

This final section is devoted to the proof of Corollary~\ref{corintro:comparison_Min}. We start with the analogue of Proposition~\ref{prop:projection} for elements $w\in W$ of finite order.

\begin{lemma}\label{lemma:proj_finite_order}
Let $w\in W$ be of finite order. Let $x\in\Min(w)$, and let $R$ be the spherical residue corresponding to $x$. Let $D\in\Ch(X)$. Then  $\dc(C,wC)\leq\dc(D,wD)$, where $C:=\proj_R(D)$. In particular, if $D\in\CMin(w)$, then $\proj_R(D)\in\CMin(w)$.
\end{lemma}
\begin{proof}
As $wR=R$, we have $wC=\proj_{R}(wD)$, so that the lemma follows from the fact that $\proj_R$ does not increase the chamber distance.
\end{proof}

\begin{prop}\label{prop:minCMin}
Let $w\in W$. Then $\Min(w)\subseteq\CMin(w)$.
\end{prop}
\begin{proof}
Let $x\in \Min(w)$. If $w$ has finite order, then by Lemma~\ref{lemma:proj_finite_order}, the projection on the residue corresponding to $x$ of any chamber $D\in\CMin(w)$ is a chamber $C\in\CMin(w)$ containing $x$. Assume now that $w$ has infinite order, and let $L$ be the $w$-axis through $x$. Let $\sigma\subseteq L$ be a nonempty open geodesic segment containing $x$ in its closure and contained in some (open) cell. Then Proposition~\ref{prop:projection} yields some chamber $D\in\CMin(w)$ containing $\sigma$, and hence also $x$.
\end{proof}

\begin{lemma}\label{lemma:centraliser_cocompact}
Let $w\in W$. Then the centraliser $C_W(w)$ of $w$ in $W$ acts cocompactly on both $\Min(w)$ and $\CMin(w)$.
\end{lemma}
\begin{proof}
The fact that $C(w):=C_W(w)$ acts cocompactly on $\Min(w)$ follows from \cite[Theorem~3.2]{Ru01}, and a straightforward adaptation of the proof of \cite[Theorem~3.2]{Ru01}, which we now provide, also shows that $C(w)$ acts cocompactly on $\CMin(w)$.

Indeed, fix some $D\in\CMin(w)$, and assume for a contradiction that there is a sequence of chambers $(D_n)_{n\in\NN}\subseteq\CMin(w)$ such that $\dc(D_n,C(w)D)\geq n$ for all $n\in\NN$. Write $D_n=v_nD$ for some $v_n\in W$. By hypothesis, we have $\dc(D,v_n\inv wv_nD)=\dc(D_n,wD_n)=\dc(D,wD)$ for all $n\in\NN$. In particular, $\{v_n\inv wv_n \ | \ n\in\NN\}$ is finite. Hence, up to extracting a subsequence, we may assume that $v_n\inv wv_n=v_0\inv wv_0$ for all $n\in\NN$, that is, $v_nv_0\inv\in C(w)$ for all $n\in\NN$. But then
$$n\leq\dc(D_n,C(w)D)\leq\dc(D_n,v_nv_0\inv D)=\dc(D,v_0\inv D)=\dc(D_0,D)$$
for all $n\in\NN$, a contradiction.
\end{proof}

\smallskip
\noindent
{\bf Proof of Corollary~\ref{corintro:comparison_Min}:} Let $w\in W$. Then $\Min(w)\subseteq\CMin(w)$ by Proposition~\ref{prop:minCMin}, and $\CMin(w)$ is at bounded Hausdorff distance from $\Min(w)$ by Lemma~\ref{lemma:centraliser_cocompact}. $\hspace{\fill}\qed$

\bibliographystyle{amsalpha} 
\bibliography{these} 

\end{document}